\documentclass{article}
\usepackage{latexsym}
\usepackage{mathrsfs}
\usepackage{indentfirst}
\usepackage{amsmath,amsthm, amssymb}
\usepackage[dvips]{graphicx}
\usepackage{graphicx}
\usepackage{epsfig}
\usepackage{verbatim}

\usepackage{multirow}

\newtheorem{theorem}{Theorem}[section]
\newtheorem{lemma}[theorem]{Lemma}

\newtheorem{problem}{Problem}

\newtheorem{conjecture}{Conjecture}

\newtheorem{claim}{Claim}
\begin{document}
\title{Minimum size of $n$-factor-critical graphs and $k$-extendable graphs\thanks{Research
supported by NSFC(10801077), Natural Science Foundation of
Guangdong Province(9451030007003340 and 9451009001002740), Natural
Science Foundation of Jiangsu Higher Education Institutions of
China(08KJB110008), and Science Research Foundation of Guangdong
Industry Technical College.}}
\author{Zan-Bo Zhang$^{1,2}$,
Xiaoyan Zhang$^{3,4}$\thanks{Corresponding Author. Email address:
royxyzhang@gmail.com}, Dingjun Lou$^2$, Xuelian Wen$^5$
\\ \small $^1$Department of Computer Engineering, Guangdong Industry Technical College, Guangzhou 510300, China.
\\ \small $^2$Department of Computer Science, Sun Yat-sen University, Guangzhou 510275, China.
\\ \small $^3$School of Mathematical Science, Nanjing Normal University, Nanjing, 210097, China.
\\ \small $^4$Faculty of Electrical Engineering, Mathematics and Computer Science, University of Twente,
\\ \small P.O. Box 217, 7500 AE Enschede, The Netherlands
\\ \small $^5$School of Economics and Management, South China Normal University, Guangzhou 510006, China.
}
\date{}
\maketitle

\begin{abstract}
We determine the minimum size of $n$-factor-critical graphs and
that of $k$-extendable bipartite graphs, by considering Harary
graphs and related graphs. Moreover, we determine the minimum size
of $k$-extendable non-bipartite graphs for $k=1,\ 2$, and pose a
related conjecture for general $k$.
$\newline$\noindent\textbf{Key words}: $k$-extendable,
$n$-factor-critical, minimum size
\end{abstract}
\section{Introduction and terminologies}
All graphs considered in this paper are finite, connected and
simple. Let $G$ be a graph with vertex set $V(G)$ and edge set
$E(G)$. The number of vertices in $G$ is denoted by $\nu(G)$ or
$\nu$. The number of edges in $G$ is called its \emph{size}. The
connectivity, the edge connectivity, the independence number, the
minimum degree and the maximum degree of $G$ are denoted by
$\kappa(G)$, $\kappa^\prime(G)$, $\alpha(G)$, $\delta(G)$ and
$\Delta(G)$, respectively. The neighborhood of a vertex $v$ in $G$
is denoted by $N_G(v)$ or $N(v)$. For other terminologies not
defined here, the reader is referred to \cite{BM1976}.

A connected graph $G$ is said to be \emph{$k$-extendable}, if it
contains a matching of size $k$ and every matching in $G$ of size
$k$ is contained in a perfect matching of $G$, where $k$ is an
integer such that $0\leq k\leq (\nu(G)-2)/2$. The concept of
$k$-extendable graphs was introduced by Plummer \cite{Plu1980}. A
graph $G$ is said to be \emph{$n$-factor-critical}, or
\emph{$n$-critical}, if $G-S$ has a perfect matching for every
$S\subseteq V(G)$ with $|S|=n$, where $0\leq n \leq \nu(G)-2$.
When $n=1$ or $2$, we say that $G$ is \emph{factor-critical} or
\emph{bicritical}. The concept of $n$-factor-critical graphs was
introduced by Yu \cite{Yu1993} and Favaron \cite{Fav1996},
independently. Extensive researches have been done on these two
classes of graphs. The reader may trace the important developments
on this field by referring to the surveys \cite{Plu1994},
\cite{Plu1996} and \cite{Plu2008} by Plummer, as well as Chapter 6
and Chapter 7 of the book \cite{YuLiu2009} by Yu and Liu.
Furthermore, a good description of the application of
extendibility and factor-criticality in job assignment can be
found in \cite{LiNie2009}.

In this paper, we determine the minimum size of a $k$-extendable
bipartite graph and that of an $n$-factor-critical graph, by
considering Harary graphs and related graphs. Then, we study the
problem of determining the minimum size of a $k$-extendable
non-bipartite graph. We solve the problem for $k=1,\ 2$, and pose
a conjecture related to the problem for general $k$.

We list some results and terminologies to be used later.

\begin{lemma} \label{lemma:Plu:k-ext}(Plummer \cite{Plu1986})
Let $G$ be a connected bipartite graph with bipartition $(U,W)$,
$k$ a positive integer such that $k\leq (\nu(G)-2)/2$. Then $G$ is
$k$-extendable if and only if $|U|=|W|$ and for all non-empty
subset $X$ of $U$ with $|X|\leq |U|-k$, $|N(X)|\geq |X|+k$.
\end{lemma}

\begin{lemma} \label{lemma:YF:k-fc}
(Yu \cite{Yu1993}, Favaron \cite{Fav1996}). A graph $G$ is
$n$-factor-critical if and only if $\nu(G)\equiv n\ (mod\ 2)$ and
for any vertex set $S\subseteq V(G)$ with $|S|\geq n$, $o(G-S)\leq
|S|-n$, where $o(G-S)$ denote the number of odd components of
$G-S$.
\end{lemma}

\begin{lemma} \label{lemma:Plu:con} (Plummer \cite{Plu1980})
A $k$-extendable graph $G$ is $(k+1)$-connected.
\end{lemma}

By Lemma \ref{lemma:Plu:con}, in a $k$-extendable graph $G$,
$\delta(G)\geq \kappa(G) \geq k+1$.

\begin{lemma} \label{lemma:Plu:ind} (Plummer \cite{Plu1980})
If $G$ is a $k$-extendable graph and if $u$ is a vertex of degree
$k+1$ in $G$, then $N(u)$ is independent.
\end{lemma}

\begin{lemma} \label{lemma:Fav:nfc-con} (Favaron \cite{Fav1996}, Liu and Yu \cite{LiuYu1998})
An $n$-factor-critical graph $G$ is $n$-connected,
$(n+1)$-edge-connected, and $\delta(G)\geq n+1$.
\end{lemma}

\begin{lemma} \label{lemma:LY:largek}
(Lou and Yu \cite{LouYu2004}) If $G$ is a $k$-extendable
non-bipartite graph on $\nu$ vertices with $k\geq \nu/4$, then
$\kappa(G)\geq 2k$.
\end{lemma}

\begin{lemma} \label{lemma:ZWL:eqv}
(Zhang et al. \cite{ZhaWanLou2007}) Let $G$ be a non-bipartite
graph on $\nu$ vertices, and $k$ an integer such that $k\geq
(\nu+2)/4$. Then $G$ is $k$-extendable if and only if it is
$2k$-factor-critical.
\end{lemma}

\begin{lemma} \label{lemma:mv1:ext_ind}
(Maschlanka and Volkmann \cite{MasVol1996}) If $G$ is a
$k$-extendable non-bipartite graph, then $\alpha(G)\leq
\nu(G)/2-k$.
\end{lemma}

A fullerene graph is a planar cubic $3$-connected graph with only
pentagonal and hexagonal faces. A fullerene graph on $n$ vertices
exists for even $n\geq 20$, except $n=22$.

\begin{lemma} \label{lemma:zz1:ful_ext}
(Zhang and Zhang \cite{ZhaZha2001}) Every fullerene graph is
$2$-extendable.
\end{lemma}

\section{Minimum size of $n$-factor-critical graphs and $k$-extendable bipartite graphs} \label{section:Harary}
In this section, we determine the minimum size of
$n$-factor-critical graphs and $k$-extendable bipartite graphs.

Let $G$ be a graph on $\nu$ vertices. We use numbers $0$, $1$,
$\ldots$, $\nu-1$ to label the $\nu$ vertices of a graph $G$. We
call a vertex labelled by an even (odd) number an \emph{even (odd)
vertex}. We denote an edge of $G$ jointing $i$ to $j$ by $(i,j)$,
and a path or cycle in $G$ by the listing all vertices on it in
order. Throughout this section, the labels of the vertices are
reduced modulo $\nu$.

In \cite{Har1969}, Harary defined Harary graphs, which are
$m$-connected graphs on $\nu$ vertices with $\lceil m\nu/2 \rceil$
edges, for $2\leq m<\nu$. A \emph{Harary graph} $H_{m,\nu}$ with
vertex set $\{0,1,\ldots, \nu-1\}$ is defined as follows.

(1) $m=2r$ is even. Two vertices $i$ and $j$ are joined if
$i-r\leq j \leq i+r$.

(2) $m=2r+1$ is odd and $\nu$ is even. $H_{2r+1,\nu}$ is
constructed from $H_{2r,\nu}$ by adding the edges $(i, i+\nu/2)$,
for all $0\leq i \leq \nu/2-1$.

(3) Both $m=2r+1$ and $\nu$ are odd. $H_{2r+1,\nu}$ is constructed
from $H_{2r,\nu}$ by adding the edges $(0,(\nu-1)/2)$,
$(0,(\nu+1)/2)$ and $(i,i+(\nu+1)/2)$ for all $1\leq i
<(\nu-1)/2$.

Firstly, We will prove some results on the factor-criticality and
extendibility of Harary graphs. We present some useful notations
and definitions below. We denote by $C$ the Hamilton cycle
$(0,1,\ldots,\nu-1,0)$ in $H_{m,\nu}$. For a vertex set $S\subset
V(G)$, we define an \emph{$S$-segment} to be the maximal segment
$P$ of $C$ such that all internal vertices of $P$ belong to $S$,
while the endvertices of $P$ belong to $V(G)\backslash S$. We say
that a component of $G-S$, containing an endvertex of $P$, be
\emph{associated} with $P$. An $S$-segment $P=(i,i+1, \ldots, j)$
is called an \emph{$S$-link} if the vertices $i$ and $j$ belong to
different components of $G-S$.

\begin{theorem} \label{theorem:m even}
Let $r\geq 2$ and $\nu>2r$ be two integers. Then $H_{2r,\nu}$ is
$(2r-1)$-factor-critical if $\nu$ is odd and
$(2r-2)$-factor-critical if $\nu$ is even.
\end{theorem}
\begin{proof}
Suppose $\nu=2s+1$ is odd, and $G=H_{2r,2s+1}$ is not
$(2r-1)$-factor-critical. By Lemma \ref{lemma:YF:k-fc}, there
exists a vertex set $S\subset V(G)$ with $|S|\geq 2r-1$, such that
$o(G-S)> |S| -(2r-1)$. By parity, $o(G-S)\geq
|S|-(2r-1)+2=|S|-2r+3 $. Let $c$ be the number of components of
$G-S$. Then $c \geq o(G-S)\geq |S|- 2r+3 \geq 2$. If $c=2$ then
all equalities must hold and $|S|=2r-1$. But this is impossible
since $G$ is $2r$-connected. So, $c\geq 3$.

By the definition of $H_{2r,2s+1}$, every $S$-link $P$ contains at
least $r$ internal vertices. Since $G-S$ has at least two
components, a component of $G-S$ must be associated with at least
two $S$-links. Hence there are at least $2c/2=c$ $S$-links.
Therefore $|S|\geq cr$. So we have
$$c\geq |S|-2r+3 \geq cr-2r+3.$$ That is, $(c-2)(1-r)\geq 1$.
However, this is impossible since $c\geq 3$ and $r\geq 2$. Hence,
$G$ must be $(2r-1)$-factor-critical.

Suppose that $\nu=2s$ is even, and $G=H_{2r,2s}$ is not
$(2r-2)$-factor-critical. By Lemma \ref{lemma:YF:k-fc}, there
exists a vertex set $S\subset V(G)$ with $|S|\geq 2r-2$, such that
$o(G-S)> |S| -(2r-2)$. By parity, $o(G-S)\geq |S|-2r+4 $. Using
the same notations and analogous analysis in the case that $\nu$
is odd, we have $c\geq 3$ and $(c-2)(1-r) \geq 2$, which are
impossible.
\end{proof}

\begin{theorem} \label{thoerem:m odd}
Let $r\geq 2$ and $\nu>2r+1$ be two integers. Then $H_{2r+1,\nu}$
is $2r$-factor-critical if $\nu$ is even and
$(2r-1)$-factor-critical if $\nu$ is odd.
\end{theorem}
\begin{proof}
The proof is similar to that of Theorem \ref{theorem:m even}.
\end{proof}

\begin{theorem} \label{theorem:m=2,3}
Let $s\geq 2$ be an integer. Then $H_{2,2s-1}$ and $H_{3,2s+1}$
are factor-critical. $H_{2,2s}$ is 1-extendable. $H_{3,2s}$ is
bicritical if $s$ is even, and $2$-extendable is $s$ is odd.
\end{theorem}
\begin{proof}
Since $H_{2,\nu}$ is a cycle of order $\nu$, $H_{2,2s-1}$ is
factor-critical and $H_{2,2s}$ is $1$-extendable. Since
$H_{2,2s+1}$ is a spanning subgraph of $H_{3,2s+1}$, $H_{3,2s+1}$
is factor-critical.


Consider $H_{3,2s}$. If $s$ is even, we can verify by definition
that $H_{3,2s}$ is bicritical. If $s$ is odd, $H_{3,2s}$ is a
bipartite graph with two parts consisting of the even vertices
 and odd vertices respectively. Denote the parts of even
vertices by $U_e$, and that of odd vertices by $U_o$.

Suppose that $H_{3,2s}$ is not $2$-extendable. Then by Lemma
\ref{lemma:Plu:k-ext}, there exist a vertex set $U \subset U_e$
with $|U|\leq s-2$, such that $|N(U)| < |U|+2$. By considering the
neighborhood of $U$ on $C$, we can see that this is impossible.
\end{proof}

By Lemma \ref{lemma:Fav:nfc-con}, if $G$ is $n$-factor-critical,
then $\delta(G)\geq n+1$. So, an $n$-factor-critical graph $G$ on
$\nu$ vertices has at least $\nu(n+1)/2$ edges. Note that
$\nu(n+1)/2$ is an integer, since $n$ and $\nu$ have the same
parity. For two odd integers $n=2r-1\geq 3$ and $\nu>n$, by
Theorem \ref{theorem:m even}, $H_{n+1,\nu}$ is an
$n$-factor-critical graph on $\nu$ vertices with $\nu(n+1)/2$
edges. For two even integers $n=2r\geq 4$ and $\nu>n$, by Theorem
\ref{thoerem:m odd}, $H_{n+1,\nu}$ is an $n$-factor-critical graph
on $\nu$ vertices with $\nu(n+1)/2$ edges. For $n=1$ and an odd
integer $\nu
>1$, by Theorem \ref{theorem:m=2,3}, $H_{2,\nu}$ is a
factor-critical graph with $\nu$ edges. For $n=2$ and an even
integer $\nu>2$, it is not hard to check that a \emph{wheel}
$W_\nu$, which is formed by connecting a single vertex to all
vertices of a cycle of length $\nu-1$, is a bicritical graph on
$\nu$ vertices, with $ 3\nu/2$ edges. Therefore, the minimum size
of an $n$-factor-critical graph on $\nu$ vertices is exactly
$\nu(n+1)/2$, for all integers $n\geq 1$ and $\nu >n$, where $n$
and $\nu$ has the same parity. $\newline$

Modifying the construction of Harary graphs slightly, we can get a
class of $k$-extendable bipartite graphs with minimum size.

Let $2\leq m \leq s$ be integers, $H^B_{m,2s}$ with vertex set
$\{0,1,\ldots,2s-1\}$, is defined as follows.

(1) $m=2r$, where $r\leq s/2$. Then vertex $i$ is adjacent to $j$,
if $i$ is even, $j$ is odd and $i-2r+1\leq j \leq i+2r-1$.

(2) $m=2r+1$, where $r\leq s/2$. Then vertex $i$ is adjacent to
$j$, if $i$ is even, $j$ is odd and $i-2r+1 \leq j \leq i+2r+1$.

It is clear that every $H^B_{m,2s}$ is a balanced bipartite graph
whose two parts are consisting of the even vertices and the odd
vertices, respectively. We denote the part consisting of the even
(odd) vertices by $V_e$ ($V_o$).
\begin{theorem} \label{theorem:HB m even}
For two integers $r\geq 1$ and $s\geq 2r$, $H^B_{2r,2s}$ is
$(2r-1)$-extendable.
\end{theorem}
\begin{proof}
Let $G=H^B_{2r,2s}$ and assume that $G$ is not
$(2r-1)$-extendable. By Lemma \ref{lemma:Plu:k-ext}, there exists
$U\subset V_e$, such that $|U|\leq s-(2r-1)$ and $|N(U)|\leq
|U|+(2r-2)$.

Define a \emph{$U$-consecutive set} $U^\prime=\{2i_0, 2i_1,\ldots,
2i_{t-1}\}$ as a maximal subset of $U$ so that $2i_l<2i_{l+1}\leq
2i_l+4r-2$, for $0\leq l\leq t-2$. Then the neighborhoods of
different $U$-consecutive sets do not intersect, and $U$ can be
uniquely divided into $U$-consecutive sets. For every
$U$-consecutive set $U^\prime=\{2i_0, 2i_1,\ldots, 2i_{t-1}\}$, if
$\{2i_0-2r+1,\ 2i_0-2r+3,\ \ldots, 2i_0-1\}\cap\{2i_{t-1}+1,\
2i_{t-1}+3,\ \ldots, 2i_{t-1}+2r-1\}\neq \emptyset$, then
$N(U^\prime)=U_o$, contradicting $|N(U)|\leq |U|+(2r-2)\leq s-1$.
Hence $\{2i_0-2r+1,\ 2i_0-2r+3,\ \ldots,
2i_0-1\}\cap\{2i_{t-1}+1,\ 2i_{t-1}+3,\ \ldots, 2i_{t-1}+2r-1\}=
\emptyset$. So $|N(U^\prime)|\geq r+r+t-1=|U^\prime|+2r-1$, and
hence $|N(U)|\geq |U|+2r-1$, contradicting $|N(U)|\leq
|U|+(2r-2)$. Therefore $G$ is $(2r-1)$-extendable.
\end{proof}

\begin{theorem} \label{theorem:HB m odd}
For two integers $r\geq 1$ and $s\geq 2r+1$, $H^B_{2r+1,2s}$ is
$(2r)$-extendable.
\end{theorem}
\begin{proof}
The proof is similar to that of Theorem \ref{theorem:HB m even}.
\end{proof}

For a $k$-extendable graph $G$, $\delta(G)\geq k+1$. Hence, for an
integer $k\geq 1$ and an even integer $\nu\geq 2k+2$, a
$k$-extendable graph $G$ on $\nu$ vertices has at least $\nu
(k+1)/2$ edges. By Theorem \ref{theorem:HB m even} and
\ref{theorem:HB m odd}, for all $2\leq m \leq s$, $H^B_{m,2s}$ are
$(m-1)$-extendable bipartite graphs having $\nu (k+1)/2$ edges.
Therefore, the minimum size of a $k$-extendable bipartite graph on
$\nu$ vertices is exactly $\nu(k+1)/2$.

\section{Minimum size of $1$-extendable non-bipartite graphs and $2$-extendable non-bipartite graphs}
In the previous section we constructed $k$-extendable bipartite
graphs with minimum size. Now we consider $k$-extendable
non-bipartite graphs with minimum size. Let $G$ be a non-bipartite
graph on $\nu$ vertices, where $\nu$ even, and $k$ be an positive
integer such that $k\geq (\nu+2)/4$. By Lemma \ref{lemma:ZWL:eqv},
if $G$ is $k$-extendable, then it is $2k$-factor-critical.
Therefore, $\delta(G)\geq 2k+1$, and $G$ has at least
$(2k+1)\nu/2$ edges, which is greater than the lower bound for
$k$-extendable bipartite graphs. Hence, we raise the following
problem.

\begin{problem} \label{problem:non-bipartite}
Let $k\geq 1$ be an integer, and $G$ a $k$-extendable
non-bipartite graph on $\nu \geq 2k+2$ vertices. What is the
minimum size of $G$?
\end{problem}

Denote such a minimum number by $\varepsilon(\nu, k)$. In this
section, we solve the problem for $k=1, \ 2$.

\begin{theorem} \label{theorem:k=1}
For an even number $\nu\geq 4$, $\varepsilon(\nu,\ 1)=\nu+2$.
\end{theorem}

\begin{proof}
We have $\delta(G) \geq 2$ in a $1$-extendable graph $G$. Hence, a
$1$-extendable graph on $\nu$ vertices has at least $\nu$ edges.
However, a connected graph with $\nu$ vertices and $\nu$ edges can
only be the cycle $C_\nu$, which is bipartite. Therefore
$\varepsilon(\nu,1)\geq \nu+1$. Take a cycle $C=v_0v_1\ldots
v_{\nu-1}v_0$, joint $v_0$ to $v_2$ and $v_1$ to $v_3$ we get a
$1$-extendable non-bipartite graph $G$. Therefore
$\varepsilon(\nu, 1)\leq \nu+2$.

Let $G$ be a $1$-extendable non-bipartite graph with $\nu$
vertices and $\nu+1$ edges. By the Handshaking Lemma, $G$ has
precisely two vertices of degree $3$, while the other vertices are
of degree $2$. Since $G$ is non-bipartite, there is an odd cycle
$Q=v_0v_1\ldots v_{2l}v_0$ in $G$, and $G-Q$ is not null. By Lemma
\ref{lemma:Plu:con}, $\kappa(G)\geq 2$. So, there is at least two
vertices on $Q$, say $v_0$ and $v_i$, where $1\leq i \leq 2l$, who
send edges to $G-Q$, and hence $d(v_0)=d(v_i)=3$. If $i$ is odd,
then $v_0v_1$ is not contained in any perfect matching of $G$. If
$i$ is even, then $v_{2l}v_0$ is not contained in any perfect
matching of $G$. These contradict that $G$ is $1$-extendable.
Hence $\varepsilon(\nu, 1)\neq \nu+1$. So $\varepsilon(\nu,
1)=\nu+2$.
\end{proof}

Now we consider $2$-extendable non-bipartite graphs. For a
$2$-extendable graph $G$, $\delta(G)\geq 3$. Hence,
$\varepsilon(\nu,2)\geq 3\nu/2$. The next theorem shows that the
bound can be achieved when $\nu$ is large. We will prove the
theorem in the rest of this section.

\begin{theorem} \label{theorem:2-ext_min_edges}
For an even integer $\nu\geq 6$, $$\varepsilon(\nu,2)=\left\{
\begin{array}
{l l}
15, & \mbox{if}\ \nu=6, \\
16, & \mbox{if}\ \nu=8, \\
19, & \mbox{if}\ \nu=10, \\
20, & \mbox{if}\ \nu=12, \\
3\nu/2, & \mbox{if}\ \nu\geq 14. \\
\end{array}\right.$$
\end{theorem}

\begin{figure}[!htbp]
\centering
\includegraphics[width=0.155\linewidth]{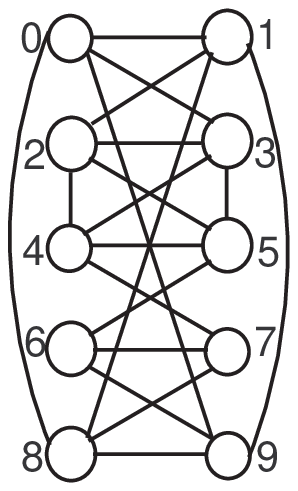}
\caption{$G_1$, a $2$-extendable non-bipartite graph with $10$
vertices and $19$ edges} \label{figure:nu10ep19}
\end{figure}

We will prove several theorems, who will be combined to obtain
Theorem \ref{theorem:2-ext_min_edges}.

\begin{theorem} \label{theorem:nu10}
$\varepsilon(10,2)=19$.
\end{theorem}
\begin{proof}
It can be verified that the graph we show in Figure
\ref{figure:nu10ep19} is a $2$-extendable non-bipartite
graph\footnote{The extendibility of the graphs that are claimed to
be $2$-extendable in this paper is verified in
\ref{section:Appendix A}.} with $10$ vertices and $19$ edges. To
prove that $\varepsilon(10,2)=19$, it suffices to show that there
do not exist $2$-extendable non-bipartite graphs with $10$
vertices and no more than $18$ edges. Assume that we can find such
a graph $G$. By Lemma \ref{lemma:mv1:ext_ind}, $\alpha(G)\leq 3$.
Firstly we prove the following claims.

\setcounter{claim}{0}
\begin{claim} \label{claim:3_ind_degree_3}
There do not exist three independent vertices of degree $3$ in
$G$.
\end{claim}

We prove Claim \ref{claim:3_ind_degree_3} by contradiction.
Suppose there are three independent vertices $u_0$, $u_1$, $u_2$
of degree $3$ in $G$. By Lemma \ref{lemma:Plu:ind}, the
neighborhood of $u_i$, $0\leq i \leq 2$, is independent. Since
$\alpha(G)\leq 3$ and $\nu(G)=10$, we have $|N(u_0)\cup N(u_1)\cup
N(u_2)|=7$.

Suppose that $|N(u_0)\cap (N(u_1)\cup N(u_2))|\leq 1$. Then
$(N(u_0)\backslash (N(u_1)\cup N(u_2))) \cup \{u_1,\ u_2\}$ is an
independent set of order at less $4$, a contradiction. So
$|N(u_0)\cap (N(u_1)\cup N(u_2))|\geq 2$. Similarly, we have
$|N(u_1)\cap (N(u_0)\cup N(u_2))|\geq 2$ and $|N(u_2)\cap
(N(u_0)\cup N(u_1))|\geq 2$. But then $|N(u_0)\cup N(u_1)\cup
N(u_2)| \leq 1+1+1+(2+2+2)/2=6$, contradicting $|N(u_0)\cup
N(u_1)\cup N(u_2)|=7$. Therefore, there do not exist three
independent vertices of degree $3$ in $G$.

\begin{claim} \label{claim:degree_4_neighbor}
Every vertex in $G$ has at most two neighbors of degree $3$.
\end{claim}

Suppose there is a vertex $u$ in $G$, who has at least three
neighbors, say $v_0$, $v_1$ and $v_2$, of degree $3$. By Claim
\ref{claim:3_ind_degree_3}, at least two of them, say $v_0$ and
$v_1$, are adjacent. Then $N(v_0)$ is not independent,
contradicting Lemma \ref{lemma:Plu:ind}. Hence Claim
\ref{claim:degree_4_neighbor} holds.

Suppose there are at least six vertices of degree $3$ in $G$, one
of which being $u$. By Claim \ref{claim:degree_4_neighbor}, $u$
can be adjacent to at most two other vertices of degree $3$. Then
there are at least three vertices of degree $3$ that are not
adjacent to $u$, at least two of which are not adjacent. Such two
vertices and $u$ are three independent vertices of degree $3$ in
$G$, contradicting Claim \ref{claim:3_ind_degree_3}. Therefore,
there are no more than five vertices of degree $3$ in $G$.
Moreover, $|E(G)|\geq \lceil(3\times 5+4\times 5)/2\rceil=18$, and
equality must hold.

Suppose there is a vertex $v$ such that $d(v)\geq 6$. By Claim
\ref{claim:degree_4_neighbor}, $v$ has at least four neighbors of
degree no less than $4$. Then $G$ has at least $\lceil(6+4\times
4+5\times 3)/2\rceil = 19$ edges, a contradiction. Hence
$\Delta(G)\leq 5$.

Now we show that $\Delta(G)<5$. Suppose that $\Delta(G)= 5$. Since
there are no more than five vertices of degree $3$, the
non-increasing degree sequence of $G$ must be $(5,\ 4,\ 4,\ 4,\
4,\ 3,\ 3,\ 3,\ 3,\ 3)$. Assume that there exist a vertex $u_0$ of
degree $3$, which is adjacent to at most one vertex of degree $3$
in $G$. There are at least other three vertices of degree $3$ that
are not adjacent to $u_0$, at least $2$ of which, denoted by $u_1$
and $u_2$, are not adjacent. Then $u_0$, $u_1$ and $u_2$ are three
independent vertices of degree $3$ in $G$, contradicting Claim
\ref{claim:3_ind_degree_3}. Hence, every vertex of degree $3$ in
$G$ has exactly two neighbors of degree $3$. Then, the five
vertices of degree $3$ in $G$ constitute a cycle, denoted by
$C_0=v_0v_1v_2v_3v_4v_0$. Furthermore, for each $0\leq i\leq 4$,
$v_i$ sends an edge to $G-C_0$.

Suppose the five vertices in $G-C_0$ are adjacent to $v_0$, $v_1$,
$v_2$, $v_3$ and $v_4$, respectively. Denote the vertex adjacent
to $v_i$ by $u_i$, $0\leq i\leq 4$. Noticing that $v_0$ and $v_2$
are not adjacent, and $u_1$, $u_3$ and $u_4$ are not adjacent to
$v_0$ or $v_2$, by $\alpha(G)\leq 3$, $u_1$, $u_3$ and $u_4$ must
be adjacent to each other. Similar analysis shows that $u_1$,
$u_2$, $u_3$, $u_4$ and $u_5$ must be mutually adjacent, a
contradiction.

Therefore, we can assume that there are two vertices on $C_0$, say
$v_0$ and $v_2$, share a common neighbor $u_0\in V(G-C_0)$. Denote
the other four vertices in $G-C_0$ by $u_1$, $u_2$, $u_3$ and
$u_4$. Since $v_0$ and $v_2$ are not adjacent, and they are not
adjacent to $u_1$, $u_2$, $u_3$ or $u_4$, by $\alpha(G)\leq 3$,
$u_1$, $u_2$, $u_3$ and $u_4$ must be mutually adjacent. Each of
$v_3$ and $v_4$ must send an edge to $\{u_1,\ u_2,\ u_3,\ u_4\}$.
Furthermore, $v_3$ and $v_4$ cannot have a common neighbor.
Without loss of generality, suppose $v_3u_3,\ v_4u_4\in E(G)$. Now
consider the independent vertices $u_0$, $v_1$ and $v_3$. Any
other vertex in $G$ must be adjacent to one of them. Therefore,
there must be an edge from $\{u_0,\ v_1\}$ to $u_4$. Similarly,
there is an edge from $\{u_0,\ v_1\}$ to $u_3$. Then we have
$d(u_3),\ d(u_4)\geq 5$, contradicting the degree sequence of $G$.
Hence, $\Delta(G)<5$.

So, all vertices in $G$ are of degree $3$ or $4$. Since
$|E(G)|=18$, a simple calculation shows that there are six
vertices of degree $4$ and four vertices of degree $3$.

\begin{figure}[!htbp]
\centering
\includegraphics[width=0.4\linewidth]{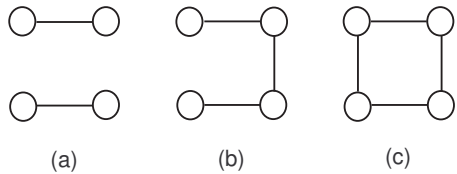}
\caption{Three possible configuration of $H$}
\label{figure:10vertices}
\end{figure}

Consider the subgraph $H$ of $G$ induced by all vertices of degree
$3$. By Lemma \ref{lemma:Plu:ind}, $H$ cannot contain a triangle.
By Claim \ref{claim:3_ind_degree_3}, $\alpha(H)\leq 2$. Hence, $H$
can only be (a), (b) or (c) in Figure \ref{figure:10vertices}. We
will discuss them case by case.

Let $H$ be (a) of Figure \ref{figure:10vertices}. Denote the two
edges in $H$ by $u_0u_1$ and $u_2u_3$. Then $u_0$, $u_1$ are
adjacent to two vertices of degree $4$ respectively, and they do
not have a common neighbor. Denote the other neighbors of $u_0$ by
$v_0$, $v_1$, those of $u_1$ by $v_2$, $v_3$. By Lemma
\ref{lemma:Plu:ind}, $v_0v_1,\ v_2v_3\notin E(G)$. Denote the
other two vertices of degree $4$ in $G$ by $v_4$ and $v_5$. Since
$\{u_0,\ v_2,\ v_3\}$ is an independent set of order $3$, $u_2$
and $u_3$ send edges to $\{v_2,\ v_3\}$. Moreover, $u_2$ and $u_3$
cannot have a common neighbor. Without loss of generality we may
assume $u_2v_2,\ u_3v_3\in E(G)$. Similarly, we can assume
$u_2v_0,\ u_3v_1\in E(G)$. Then $\{u_0,\ u_2,\ v_3\}$, $\{u_0,\
u_3,\ v_2\}$, $\{u_1,\ u_2,\ v_1\}$ and $\{u_1,\ u_3,\ v_0\}$ are
independent sets of order $3$. Hence $v_0$, $v_1$, $v_2$ and $v_3$
must be adjacent to $v_4$ and $v_5$. But then $v_4v_5\notin E(G)$,
and $\{u_0,\ u_2,\ v_4,\ v_5\}$ is an independent set of order
$4$, contradicting $\alpha(G)\leq 3$.

Let $H$ be (b) of Figure \ref{figure:10vertices}, a path
$u_0u_1u_2u_3$. Denote the other neighbors of $u_0$ by $v_0$ and
$v_1$. Then $\{u_1,\ v_0,\ v_1\}$ is an independent set. By
$\alpha(G)\leq 3$, $u_3$ must be adjacent to $v_0$ or $v_1$.
Without loss of generality, let $u_3v_0\in E(G)$. Suppose that
$u_3v_1\in E(G)$. Let $w$ be the neighbor of $v_0$ which is
different from $u_0$ and $u_3$, obviously $w\notin \{u_1,\ u_2\}$.
Then $\{u_1u_2,\ v_0w\}$ is not contained in any perfect matching
of $G$, contradicting $2$-extendibility of $G$. Hence,
$u_3v_1\notin E(G)$. So $u_3$ has another neighbor $v_2$. Denote
the other three vertices of degree $4$ in $G$ by $v_3$, $v_4$ and
$v_5$. Since $u_0$ and $u_3$ are not adjacent and they send no
edge to $\{v_3, v_4, v_5\}$, $v_3$, $v_4$ and $v_5$ must form a
triangle or we get an independent set of order $4$ in $G$. Since
$\{u_0,\ u_2,\ v_2\}$ is an independent set of order $3$, each of
$v_3$, $v_4$ and $v_5$ sends an edge to $u_2$ or $v_2$. Similarly,
each of $v_3$, $v_4$ and $v_5$ sends an edges to $u_1$ or $v_1$.
Then $v_3$, $v_4$ and $v_5$ cannot send any edge to $v_0$. So,
$N(v_0)=\{u_0, u_3\}$, a contradiction.

Let $H$ be (c) of Figure \ref{figure:10vertices}, a cycle
$u_0u_1u_2u_3u_0$. Suppose that $u_0$ and $u_2$ have a common
neighbor other than $u_1$ and $u_3$. Then the other five vertices
of degree $4$ are not adjacent to $u_0$ or $u_2$, so they must be
mutually adjacent or we get an independent set of order $4$ in
$G$. But then $G$ is not connected, a contradiction. Therefore,
$N(u_0) \cap N(u_2)=\{u_1, u_3\}$. Similarly, $u_1$ and $u_3$ do
not have any common neighbor other than $u_0$ and $u_2$. Hence,
each $u_i$ has a neighbor $v_i \notin \{u_0,\ u_1,\ u_2,\ u_3\}$,
and $v_i\neq v_j$ for $i\neq j$, $0\leq i,j \leq 3$. Denote by
$v_4$ and $v_5$ the other two vertices left. If there are any two
vertices in $\{v_0,\ v_2,\ v_4,\ v_5\}$ that are not adjacent,
then they form an independent set of order $4$ with $u_1$ and
$u_3$, a contradiction. Hence $v_0$, $v_2$, $v_4$ and $v_5$ are
mutually adjacent. Similarly $v_1$, $v_3$, $v_4$ and $v_5$ are
mutually adjacent. But then $d(v_4),\ d(v_5)\geq 5$, a
contradiction.

Thus we have led to contradictions in all cases and deny the
existence of a $2$-extendable non-bipartite graph with $10$
vertices and no more than $18$ edges, and conclude that
$\varepsilon(10,2)=19$.
\end{proof}

\begin{figure}[!htbp]
\centering
\includegraphics[width=0.27\linewidth]{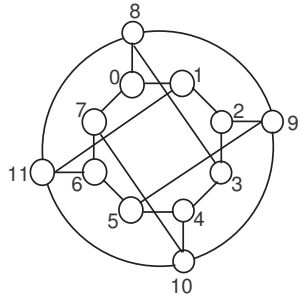}
\caption{$G_2$, the unique $2$-extendable non-bipartite graphs
with $12$ vertices and $20$ edges} \label{figure:nu12ep20}
\end{figure}

\begin{theorem} \label{theorem:nu12}
$\varepsilon(12,2)=20$. Furthermore, there is only one
$2$-extendable non-bipartite graphs with $12$ vertices and $20$
edges up to isomorphism, as showed in Figure
\ref{figure:nu12ep20}.
\end{theorem}
\begin{proof}
It can be checked that the graph showed in Figure
\ref{figure:nu12ep20} is $2$-extendable. Let $G$ be a
$2$-extendable non-bipartite graph with $12$ vertices and no more
than $20$ edges, we prove that $G$ must be isomorphic to the graph
showed in Figure \ref{figure:nu12ep20}.

By Lemma \ref{lemma:mv1:ext_ind}, $\alpha(G)\leq4$. Let the number
of vertices of degree $3$ in $G$ be $x$, then
$(3x+4(12-x))/2\leq|E(G)|\leq 20$, hence $x\geq 8$. We discuss two
cases.

\noindent\textbf{Case 1. } There are four independent vertices
$u_0$, $u_1$, $u_2$ and $u_3$, of degree $3$ in $G$.

By $\alpha(G)\leq 4$ and $\nu(G)=12$, we have $|N(u_0)\cup N(u_1)
\cup N(u_2) \cup N(u_3)|= 8$.

By Lemma \ref{lemma:Plu:ind}, the neighbors of $u_0$ are
independent. If there are two neighbors $v_0$ and $v_1$ of $u_0$
that are not adjacent to $u_1$, $u_2$ or $u_3$, then $\{v_0,\
v_1,\ u_1,\ u_2,\ u_3\}$ is an independent set, contradicting
$\alpha(G)\leq 4$. So $|N(u_0)\cap(\cup_{1\leq j\leq
3}N(u_j))|\geq 2$.
Similarly, for any $0\leq i\leq 3$, $|N(u_i)\cap(\cup_{0\leq j\leq
3,\ j\neq i}N(u_j))|\geq 2$.
Hence, $|N(u_0)\cup N(u_1) \cup N(u_2) \cup N(u_3)| \leq 1\times
4+2\times 4/2=8$. Since equality holds, for every $u_i$,
$|N(u_i)\cap(\cup_{0\leq j\leq 3,\ j\neq i}N(u_j))|= 2$.
Furthermore, any vertices in $V(G)\backslash \{u_0,\ u_1,\ u_2,\
u_3\}$ can be adjacent to at most two vertices in $\{u_0,\ u_1,\
u_2,\ u_3\}$.

Suppose $|N(u_0)\cap N(u_1)|= 2$. Then $N(u_0)\cup \{u_2,\ u_3\}$
is an independent set of order $5$, a contradiction. Hence
$|N(u_0)\cap N(u_1)|\leq 1$, and similarly $|N(u_i)\cap
N(u_j)|\leq 1$ for any $0\leq i\neq j\leq 3$. So, for every $u_i$,
$0\leq i \leq 3$, $u_i$ has common neighbors with $u_j$ and $u_k$,
where $0\leq j\neq k \leq 3$, and $j,k \neq i$. Without loss of
generality, suppose that $u_0$ has common neighbors with $u_1$ and
$u_3$. Then $u_2$ also has common neighbors with $u_1$ and $u_3$.
Hence, $u_0$, $u_1$, $u_2$, $u_3$ and the common neighbors form a
cycle on $8$ vertices. Denote the cycle by
$C_1=u_0v_0u_1v_1u_2v_2u_3v_3u_0$ and the other vertices adjacent
to $u_i$ by $w_i$, $0\leq i\leq3$.

If $w_0w_1\notin E(G)$, $\{w_0,\ w_1,\ v_0,\ u_2,\ u_3\}$ is an
independent set of order $5$, a contradiction. So $w_0w_1\in
E(G)$. Similarly, $w_1w_2, w_2w_3, w_3w_0 \in E(G)$.

If $v_0v_2\in E(G)$, then $\{v_0v_2, w_1w_2\}$ is not contained in
any perfect matching of $G$, a contradiction. Therefore
$v_0v_2\notin E(G)$. Furthermore, $v_0$ can not be adjacent to
$v_1$ or $v_3$. Similarly, all $v_i$, $0\leq i\leq 3$, cannot be
adjacent to each other. Hence every $v_i$, $0\leq i\leq 3$, sends
edges to some $w_j$, $0\leq j \leq 3$, and the number of such
edges is at least $4$. Then, $E(G) \geq 8+4+4+4=20$. By our
assumption, equality holds, and each $v_i$, $0\leq i\leq 3$, sends
exactly one edge to $w_j$, for one $0\leq j \leq3$.

The vertex $v_0$ can only be adjacent to $w_2$ or $w_3$. Without
lose of generality, suppose $v_0w_3\in E(G)$. The vertex $v_1$ can
only be adjacent to $w_0$ or $w_3$. If $v_1w_3\in E(G)$, then
$\{u_2v_2,\ u_0v_3\}$ is not contain in any perfect matching of
$G$, a contradiction. Therefore, we must have $v_1w_0\in E(G)$.
Similarly $v_2w_1,\ v_3w_2\in E(G)$. So, $G$ is isomorphic to the
graph showed in Figure \ref{figure:nu12ep20}.

\noindent \textbf{Case 2.} There do not exist four independent
vertices of degree $3$ in $G$.

We claim that there exists a vertex of degree $3$, whose neighbors
are all of degree $3$. Suppose to the contrary that there is no
such a vertex. Obviously, there exists a vertex $u$ of degree $3$,
who has two neighbors, denoted by $v$ and $w$, of degree $3$ in
$G$. Then, each of $v$ and $w$ has at most one more neighbor of
degree $3$. So, there are at least three vertices of degree $3$ in
$G$ who are not adjacent to $v$ or $w$. By the condition of Case
2, these three vertices must form a triangle, a contradiction to
Lemma \ref{lemma:Plu:ind}. So there is a vertex, say $u_0$, of
degree $3$ in $G$, whose neighbors are all of degree $3$.

By Lemma \ref{lemma:Plu:ind}, $N(u_0)$ is independent. By the
condition of Case 2, any other vertex of degree $3$ must be
adjacent to some vertices in $N(u_0)$. Hence, there is a neighbor
$u_1$ of $u_0$, who is adjacent to other two vertices of degree
$3$. Denote the other neighbors of $u_0$ by $v_0$ and $v_1$, and
the other neighbors of $u_1$ by $v_2$ and $v_3$. There are at
least two more vertices, say $w_0$ and $w_1$, of degree $3$ in
$G$.

Since there are no four independent vertices of degree $3$, there
is at least one edge among $v_0$, $v_1$, $v_2$ and $v_3$. But
$v_0v_1,\ v_2v_3\notin E(G)$. Without lose of generality we assume
that $v_1v_2\in E(G)$. Since $\{u_0,\ v_2,\ v_3\}$ is an
independent set, by the condition of Case 2, both $w_0$ and $w_1$
send some edges to $\{v_2,\ v_3\}$. Similarly, both $w_0$ and
$w_1$ send some edges to $\{v_0,\ v_1\}$.

if $w_0w_1\in E(G)$, the subgraph of $G$ induced by $\{u_0,\ u_1,\
v_0,\ v_1,\ v_2,\ v_3,\ w_0,\ w_1\}$ sends at most two edges to
the other part of $G$, so $\kappa(G)\leq \kappa^\prime(G)\leq 2$,
contradicting Lemma \ref{lemma:Plu:con}. Hence $w_0w_1\notin
E(G)$.

If $v_0$ is not adjacent to $w_0$ or $w_1$, then $u_1$, $v_0$,
$w_0$, $w_1$ are four independent vertices of degree $3$,
contradicting the condition of Case 2. So $v_0$, and similarly
$v_1$, $v_2$ and $v_3$, must be adjacent to $w_0$ or $w_1$. Since
both $w_0$ and $w_1$ send some edges to $\{v_0,\ v_1\}$ and
$\{v_2,\ v_3\}$, without lose of generality, we can assume that
$w_0v_0,\ w_1v_1\in E(G)$. Then $w_1$ must be adjacent to $v_3$,
and $v_2$ must be adjacent to $w_0$. But then $\{u_0v_0, w_1v_3\}$
is not contained in any perfect matching of $G$, contradicting
$2$-extendibility of $G$.

Therefore, there can be only one $2$-extendable non-bipartite
graph with $12$ vertices and no more than $20$ edges, upto
isomorphism, as shown in Figure \ref{figure:nu12ep20}.
\end{proof}

\begin{theorem} \label{theorem:nu14to18}
When $\nu=14$, $16$ or $18$, $\varepsilon(\nu,2)=3\nu/2$.
\end{theorem}
\begin{proof}
We already have $\varepsilon(\nu,2)\geq3\nu/2$. To prove the
equality, we show $2$-extendable graphs on $\nu=14$, $16$ and $18$
vertices and $3\nu/2$ edges in Figure \ref{figure:nu14ep21},
\ref{figure:nu16ep24} and \ref{figure:nu18ep27}, respectively.
\end{proof}

\begin{theorem} \label{theorem:nu20}
When $\nu\geq 20$, $\varepsilon(\nu,2)=3\nu/2$.
\end{theorem}

\begin{proof}
We have $\varepsilon(\nu,2)\geq 3\nu/2$. To prove the equality we
must find $3$-regular $2$-extendable non-bipartite graphs on $\nu$
vertices for all even $\nu\geq 20$. By definition and Lemma
\ref{lemma:zz1:ful_ext}, Fullerene graphs are $3$-regular
$2$-extendable non-bipartite graphs, and Fullerene graphs with
$\nu$ vertices exists for all even $\nu \geq 20$, except $\nu=22$.
So, we only need to construct a $3$-regular $2$-extendable
non-bipartite graph on $22$ vertices. One such graph is shown in
Figure \ref{figure:nu22ep33}.
\end{proof}

Now we can prove Theorem \ref{theorem:2-ext_min_edges}.
\begin{proof}
Let $G$ be a $2$-extendable graph on $\nu$ vertices with minimum
size. By Theorem \ref{lemma:LY:largek}, when $\nu\leq 8$,
$\delta(G) \geq \kappa(G) \geq 4$. For $\nu=6$, it is not hard to
check that $G$ must be $K_6$, thus $\varepsilon(6,2)=15$. For
$\nu=8$, we have $\varepsilon(8,2)\geq 16$, and it is obvious that
 the graph shown in Figure \ref{figure:nu8ep16} is a $2$-extendable non-bipartite graph with $8$ vertices and $16$
edges, hence $\varepsilon(8,2)=16$. For other values of $\nu$ the
results follow from Theorem \ref{theorem:nu10},
\ref{theorem:nu12}, \ref{theorem:nu14to18} and \ref{theorem:nu20}.
\end{proof}

\begin{figure}[!htbp]
\centering
\includegraphics[width=0.26\linewidth]{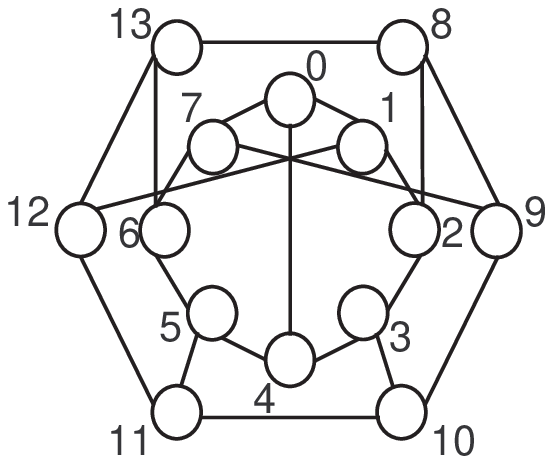}
\caption{$G_3$, a $2$-extendable non-bipartite graph with $14$
vertices and $21$ edges} \label{figure:nu14ep21}
\end{figure}

\begin{figure}[!htbp]
\centering
\includegraphics[width=0.25\linewidth]{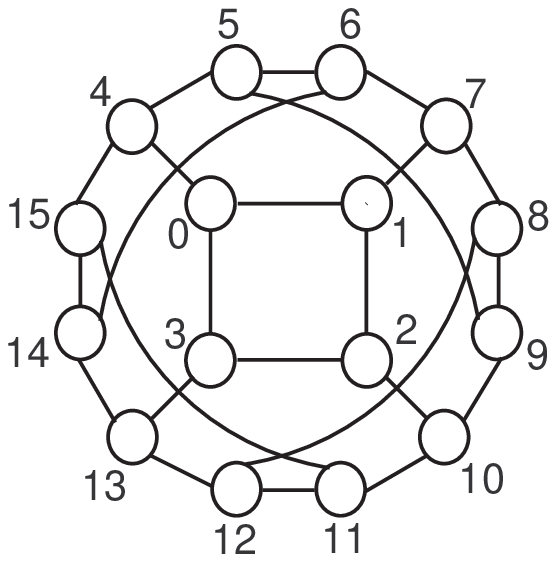}
\caption{$G_4$, a $2$-extendable non-bipartite graph with $16$
vertices and $24$ edges} \label{figure:nu16ep24}
\end{figure}

\begin{figure}[!htbp]
\centering
\includegraphics[width=0.26\linewidth]{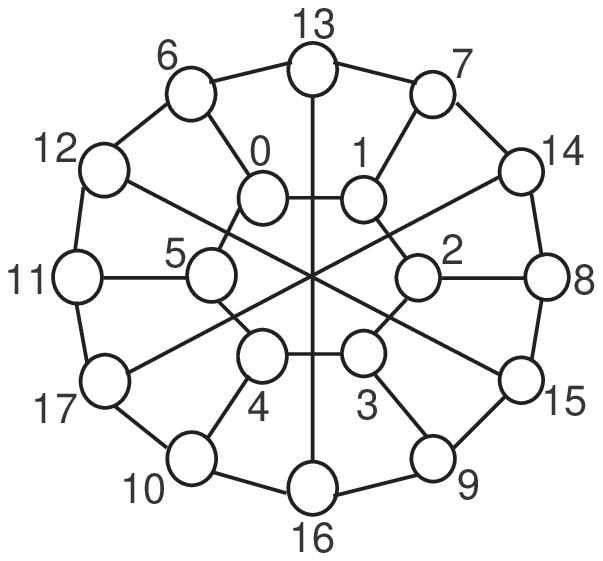}
\caption{$G_5$, a $2$-extendable non-bipartite graph with $18$
vertices and $27$ edges} \label{figure:nu18ep27}
\end{figure}

\begin{figure}[!htbp]
\centering
\includegraphics[width=0.31\linewidth]{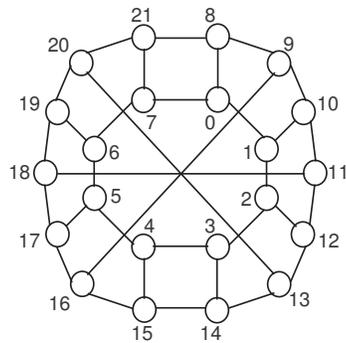}
\caption{$G_6$, a $2$-extendable non-bipartite graph with $22$
vertices and $33$ edges} \label{figure:nu22ep33}
\end{figure}

\begin{figure}[!htbp]
\centering
\includegraphics[width=0.15\linewidth]{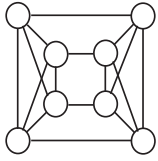}
\caption{a $2$-extendable non-bipartite graph with $8$ vertices
and $16$ edges} \label{figure:nu8ep16}
\end{figure}

\newpage

\section{Final remarks}
We finish our paper with some ideas on Problem
\ref{problem:non-bipartite} for general $k$.

By Lemma \ref{lemma:ZWL:eqv}, the set of $k$-extendable
non-bipartite graphs and the set of $2k$-factor-critical graphs
coincide when $\nu(G)\leq 4k-2$. We have found
$2k$-factor-critical graphs with minimum size among Harary graphs
in Section \ref{section:Harary}. So we have
$$\varepsilon (\nu, k) =
(2k+1)\nu/2 \hspace{2em}\mbox{if}~~ \nu \le 4k-2.$$

By Lemma \ref{lemma:LY:largek}, when $\nu(G)=4k$, the connectivity
of a $k$-extendable non-bipartite graph $G$ is $2k$. Therefore
$\delta(G)\geq \kappa(G)\geq 2k$. This bound is obtained by the
following graph. Let $H_1$ and $H_2$ be two copies of $K_{2k}$
where $V(H_1)=\{u_0,\ u_1,\ \ldots, u_{2k-1}\}$ and
$V(H_2)=\{v_0,\ v_1,\ \ldots,\ v_{2k-1}\}$. And construct $G$ by
joining every $u_i$ to $v_i$, $0\leq i \leq 2k-1$. It is not hard
to check that $G$ is a $k$-extendable graph with $4k$ vertices and
regular degree $2k$. Therefore
$$\varepsilon(4k, k)=4k^2.$$

When $\nu \geq 4k+2$, an example in \cite{LouYu2004} shows that
the connectivity of a $k$-extendable graph $G$ with $\nu$ vertices
can be $k+1$, that is, the bound given by Lemma
\ref{lemma:Plu:con}. The case that $k=2$ gives us some hints that
there may exist $k$-extendable regular graphs with degree $k+1$,
and hence $\varepsilon(\nu, k)=(k+1)\nu/2$, when $\nu$ is large.
Let $\nu_0$ be the minimum even integer such that
$\varepsilon(\nu_0,k)=(k+1)\nu_0/2$. Then $\nu_0 \geq 4k+2$.
Assuming that for a given $k$, the function $\varepsilon(\nu, k)$
is increasing for even integer $\nu$. We have $(\nu_0-4k)/2\leq
\varepsilon(\nu_0,k)-\varepsilon(4k,k) = (k+1)\nu_0/2-4k^2$, that
is, $\nu_0 \geq 8k-4$. Therefore, we have the following
conjecture.

\begin{conjecture}
For a given integer $k>0$, if $G$ is a $k$-extendable
non-bipartite graph with edge number $\nu(G)(k+1)/2$, then
$\nu(G)\geq 8k-4$.
\end{conjecture}

\section*{Acknowledgments}
We would like to thank the referees for their careful reading and
many valuable suggestions that have greatly improved the
presentation of the paper.

\newpage

\appendix

\renewcommand\thesection{\appendixname~\Alph{section}}

\renewcommand\thesubsection{\Alph{section}.\arabic{subsection}}

\renewcommand\thetable{\Alph{section}.\arabic{table}}

\section{The verification of $2$-extendibility} \label{section:Appendix
A}

In this appendix, we verify the $2$-extendibility of the graphs in
Figure \ref{figure:nu10ep19}, Figure \ref{figure:nu12ep20}, Figure
\ref{figure:nu14ep21}, Figure \ref{figure:nu16ep24}, Figure
\ref{figure:nu18ep27} and Figure \ref{figure:nu22ep33}, which have
been named as $G_i$, $1\leq i \leq 6$. By definition, we check
that every non-adjacent edge pair in $G_i$ is contained in a
perfect matching of it. Making advantage of the symmetry and the
cycle structures of the graphs, we only need to check a small part
of all cases.

We verify the extendibility of each $G_i$ in a separate section.
For convenience, we label the $\nu$ vertices of a graph by $0$,
$1$, $\ldots$, $\nu-1$, as shown in the figures. An edge with
endvertices $i$ and $j$ is denoted by $i$-$j$ when we list it in
the tables, and $(i,j)$ in other contents.

\subsection{$2$-extendibility of $G_1$}

Let $E_0=\{e_0,\ e_1\}$ be an edge pair in $G_1$ and let
$G_1^\prime\ =G_1- \{(0,8),\ (2,4),\ (1,9),\ (3,5)\}$. Then
$G_1^\prime$ is isomorphic to $H^B_{3,10}$ defined in Section
\ref{section:Harary}, which is $2$-extendable. Hence, if $E_0 \cap
\{ (0,8),\ (2,4),\ (1,9),\ (3,5)\}\\=\emptyset$, then $E_0$ is
contained in a perfect matching of $G_1^\prime$, which is also a
perfect matching of $G_1$.

If $E_0 \subset \{(0,8),\ (2,4),\ (1,9),\ (3,5)\}$, then $E_0$ is
contained in the perfect matching $\{(0,8),\ (2,4),\ (1,9),\\
(3,5),\ (6,7)\}$.

What left is the case that $|E_0\cap \{(0,8),\ (2,4),\ (1,9),\
(3,5)\}|=1$. By symmetry, we only need to verify that every such
edge pair containing $(0,8)$ or $(2,4)$ is contained in a perfect
matching of $G_1$.

To verify that some edge pairs can be extended to perfect
matchings, we list several perfect matchings, so that each edge
pair is contained in at least one of them.

We arrange the verification data in tables. The first column of a
table contains the edges. The second column of a table contains
the edge pairs containing the edge. For convenience, we just list
the other edge in each edge pair. The third column lists several
perfect matchings so that each edge pair in the second column is
contained in at least one of them.

\begin{table}[!h]
\caption{$2$-extendibility of $G_1$} \centering
\begin{tabular}{|l|l|l|}
\hline
Edge & Edge Pairs & Perfect Matchings \\

\hline \multirow{2}{*} {0-8} & {1-2, 2-3, 2-5, 3-4, 4-5,} &
{\{0-8,1-2,3-5,4-7,6-9\} \{0-8,1-9,2-3,4-5,6-7\}} \\
& {4-7, 5-6, 6-7, 6-9.} & {\{0-8,1-9,2-5,3-4,6-7\} \{0-8,1-9,2-3,4-7,5-6\}} \\

\hline \multirow{2}{*} {2-4} & {0-1, 0-3, 0-9, 1-8, 5-6,} &
\{2-4,0-1,3-5,6-7,8-9\} \{2-4,0-3,1-9,5-6,7-8\}\\
& {6-7, 6-9, 7-8, 8-9.} & \{2-4,0-9,1-8,3-5,6-7\} \{2-4,0-1,3-5,6-9,7-8\}\\
\hline
\end{tabular}
\end{table}

\subsection{$2$-extendibility of $G_2$}
Let $E_0=\{e_0,\ e_1\}$ be an edge pair in $G_2$. We verify that
$E_0$ is contained in a perfect matching of $G_2$. Let $C_1=(0,\
1,\ 2,\ 3,\ 4,\ 5,\ 6,\ 7,\ 0)$ and $(C_2=8,\ 9,\ 10,\ 11,\ 8)$ be
the two cycles in $G_2$. If $|E_0\cap E(C_1)|=1$ and $|E_0\cap
E(C_2)|=1$, then $E_0$ can be easily extended to a perfect
matching of $G_2$, which is composed of a perfect matching of
$C_1$ and a perfect matching of $C_2$. Hence, we can assume that
$|E_0\cap E(C_1)|=0$ or $|E_0\cap E(C_2)|=0$.

Let $e=v_0v_1$ and $f=v_2v_3$ be two edges on an even cycle $C$,
where $v_0$, $v_1$, $v_2$ and $v_3$ appear on $C$ as the order
listed. Since $C$ is even, the length of the segment of $C$ from
$v_1$ to $v_2$ and that of the segment of $C$ from $v_2$ to $v_1$
have the same parity. We say that two vertices $v_1$ and $v_2$ are
at an odd distance on $C$, if the length is odd. Furthermore, we
say that $e$ and $f$ are at an odd distance on $C$, if $v_1$ and
$v_2$ are at an odd distance on $C$. Note that by parity, $v_0$
and $v_3$ is also at an odd distance on $C$. Obviously, any edge
pair whose two edges are at an odd distance on $C_1$ or $C_2$ can
be easily extended to a perfect matching of $G$. So we skip these
edge pairs in our verification lists. Similar skipping is applied
in subsequent sections.

Firstly, we examine all edge pairs that contain an edge on $C_1$
but no edge on $C_2$. By symmetry, we only need to examine all
such edge pairs containing the edge $(0,1)$ or $(1,2)$. Then, we
examine all edge pairs that contain no edge on $C_1$ but an edge
on $C_2$. Again by symmetry, it suffices to examine all such edge
pairs containing $(8,9)$. Finally, we examine the edge pairs that
contains no edge on $C_1$ or $C_2$. By symmetry, we examine such
edge pairs containing $(0,8)$ or $(1,11)$.

\newpage

\begin{table}[!h]
\caption{$2$-extendibility of $G_2$} \centering
\begin{tabular}{|l|l|l|}
\hline
Edge & Edge Pairs & Perfect Matchings \\

\hline \multirow{2}{*} {0-1} & {3-4, 5-6, 2-9, 4-10, 6-11,} &
\{0-1,2-9,3-4,5-6,7-10,8-11\} \{0-1,2-9,3-8,4-5,6-7,10-11\} \\
& {3-8, 5-9, 7-10.} & \{0-1,2-3,4-10,5-9,6-7,8-11\} \{0-1,2-3,4-5,6-11,7-10,8-9\} \\

\hline \multirow{2}{*} {1-2} & {4-5, 6-7, 0-8, 4-10, 6-11,} &
\{1-2,0-8,3-4,5-9,6-11,7-10\} \{1-2,0-7,3-8,4-5,6-11,9-10\} \\
& {3-8, 5-9, 7-10.} & \{1-2,0-7,3-8,4-10,5-9,6-11\} \{1-2,0-8,3-4,5-9,6-7,10-11\}\\

\hline {8-9} & {6-11, 4-10, 1-11, 7-10.} & \{8-9,0-7,1-11,2-3,4-10,5-6\} \{8-9,0-1,2-3,4-5,6-11,7-10\} \\

\hline \multirow{2}{*} {0-8} & {2-9, 4-10, 6-11, 1-11, 5-9,}
& {\{0-8,1-11,2-9,3-4,5-6,7-10\} \{0-8,1-11,2-3,4-10,5-9,6-7\}} \\
& {7-10.} & {\{0-8,1-2,3-4,5-9,6-11,7-10\}} \\

\hline \multirow{2}{*} {1-11} & {0-8, 2-9, 4-10, 3-8, 5-9,} & \{1-11,0-7,2-9,3-8,4-10,5-6\} \{1-11,0-8,2-3,4-10,5-9,6-7\} \\
& {7-10.} & \{1-11,0-8,2-9,3-4,5-6,7-10\} \\
\hline
\end{tabular}
\end{table}

\subsection{$2$-extendibility of $G_3$} \label{G_3}
Let $E_0=\{e_0,\ e_1\}$ be an edge pair in $G_3$. We verify that
$E_0$ is contained in a perfect matching of $G_3$. Let $C_1=(0,\
1,\ 2,\ 3,\ 4,\ 5,\ 6,\ 7,\ 0)$ and $C_2=(8,\ 9,\ 10,\ 11,\ 12,\
13,\ 8)$ be the two cycles in $G_3$. If $|E_0\cap E(C_1)|=1$ and
$|E_0\cap E(C_2)|=1$, then $E_0$ can be extended to a perfect
matching of $G_3$, which is composed of a perfect matching of
$C_1$ and a perfect matching of $C_2$. So we can assume that
$|E_0\cap E(C_1)|=0$ or $|E_0\cap E(C_2)|=0$. Firstly, we examine
all edge pairs containing an edge on $C_1$ but no edge on $C_2$.
By symmetry, we examine all such edge pairs containing $(0,1)$,
$(1,2)$ , $(2,3)$ or $(3,4)$. Then, we examine all edge pairs that
contain no edge on $C_1$, but an edge on $C_2$. By symmetry, we
examine all such edge pairs containing $(8,9)$, $(8,13)$, $(9,10)$
or $(10,11)$. Finally, if $E_0$ contains no edge on $C_1$ or
$C_2$, then $E_0$ must be contained in the perfect matching
$\{(0,4),\ (1,12),\ (2,8),\ (3,10),\ (5,11),\ (6,13),\ (7,9)\}$.

\begin{table}[!h]
\caption{$2$-extendibility of $G_3$} \centering
\begin{tabular}{|l|l|l|}
\hline
Edge & Edge Pairs & Perfect Matchings \\

\hline \multirow{2}{*} {0-1} & {3-4, 5-6, 2-8, 3-10, 5-11,} &
\{0-1,2-8,3-4,5-6,7-9,10-11,12-13\}
\{0-1,2-8,3-10,4-5,6-13,7-9,11-12\}\\
& {6-13, 7-9} &\{0-1,2-8,3-4,5-11,6-7,9-10,12-13\} \\

\hline \multirow{2}{*} {1-2} & {4-5, 6-7, 0-4, 3-10, 5-11,} &
\{1-2,0-7,3-10,4-5,6-13,8-9,11-12\}
\{1-2,0-4,3-10,5-6,7-9,8-13,11-12\}\\
& {6-12, 7-13} &\{1-2,0-4,3-10,5-11,6-7,8-9,12-13\} \\

\hline \multirow{2}{*} {2-3} & {0-7, 5-6, 0-4, 1-12, 5-11,} &
\{2-3,0-4,1-12,5-6,7-9,8-13,10-11\}
\{2-3,0-7,1-12,4-5,6-13,8-9,10-11\}\\
& {6-13, 7-9} &\{2-3,0-4,1-12,5-11,6-7,8-13,9-10\} \\

\hline \multirow{2}{*} {3-4} & {0-1, 6-7, 1-12, 2-8, 5-11,} &
\{3-4,0-1,2-8,5-6,7-9,10-11,12-13\}
\{3-4,0-7,1-12,2-8,5-11,6-13,9-10\}\\
& {6-13, 7-9} &\{3-4,0-1,2-8,5-11,6-7,9-10,12-13\} \\

\hline \multirow{2}{*} {8-9} & {11-12, 0-4, 1-12, 3-10, 5-11,} &
\{8-9,0-4,1-2,3-10,5-11,6-7,12-13\}
\{8-9,0-7,1-12,2-3,4-5,6-13,10-11\}\\
& {6-13} &\{8-9,0-7,1-2,3-10,4-5,6-13,11-12\} \\

\hline \multirow{2}{*} {8-13} & {10-11, 0-4, 1-12, 3-10, 5-11,} &
\{8-13,0-4,1-2,3-10,5-6,7-9,11-12\}
\{8-13,0-4,1-12,2-3,5-11,6-7,9-10\}\\
& {7-9} &\{8-13,0-4,1-12,2-3,5-6,7-9,10-11\} \\

\hline \multirow{2}{*} {9-10} & {12-13, 0-4, 1-12, 2-8, 5-11,} &
\{9-10,0-4,1-12,2-3,5-11,6-7,8-13\}
\{9-10,0-7,1-12,2-8,3-4,5-11,6-13\}\\
& {6-13} &\{9-10,0-1,2-8,3-4,5-11,6-7,12-13\} \\

\hline \multirow{2}{*} {10-11} & {8-13, 0-4, 1-12, 2-8, 6-13,} &
\{10-11,0-4,1-12,2-3,5-6,7-9,8-13\}
\{10-11,0-1,2-8,3-4,5-6,7-9,12-13\}\\
& {7-9} &\{10-11,0-7,1-12,2-3,4-5,6-13,8-9\} \\

\hline
\end{tabular}
\end{table}

\subsection{$2$-extendibility of $G_4$}

Let $C_1=(0,\ 1,\ 2,\ 3,\ 0)$ and $C_2=(4,\ 5,\ 6,\ 7,\ 8,\ 9,\
10,\ 11,\ 12,\ 13,\ 14,\ 15,\ 4)$. The analysis is analog to that
in Section \ref{G_3}.

\subsection{$2$-extendibility of $G_5$}
Let $C_1=(0,\ 1,\ 2,\ 3,\ 4,\ 5,\ 0)$ and $C_2=(6,\ 13,\ 7,\ 14,\
8,\ 15,\ 9,\ 16,\ 10,\ 17,\ 11,\ 12,\ 6)$. The analysis is analog
to that in Section \ref{G_3}.

\subsection{$2$-extendibility of $G_6$}

Let $C_1=(0,\ 1,\ 2,\ 3,\ 4,\ 5,\ 6,\ 7,\ 0)$ and $C_2=(8,\ 9,\
10,\ 11,\ 12,\ 13,\ 14,\ 15,\ 16,\ 17,\ 18,\ 19,\ 20,\ 21,\ 8)$.
The analysis is analog to that in Section \ref{G_3}.


\begin{thebibliography}{99}
\bibitem{BM1976} Bondy, J.A., Murty, U.S.R.: Graph Theory with
Applications. The Macmillan Press, London (1976).
\bibitem{Fav1996} Favaron, O.: On $k$-factor-critical graphs. Discuss.
Math. Graph Theory 16, 41-51 (1996).
\bibitem{Har1969} Harary, F.: The maximum connetivity of a graph.
Proc. Nat. Acad. Sci. U.S.A. 48, 1142-1164 (1962).
\bibitem{LiNie2009} Li, Y., Nie, Z.: A note on n-critical
bipartite graphs and its application. in: Du, D., Hu, X.,
Pardalos, P.M. (eds.) COCOA 2009, LNCS 5573, pp. 279-286. Springer
(2009).
\bibitem{LiuYu1998} Liu, G., Yu, Q.: On $n$-edge-deletable and
$n$-critical graphs. Bull. Inst. Combin. Appl. 24, 65-72 (1998).
\bibitem{LouYu2004} Lou, D., Yu, Q.: Connectivity of
$k$-extendable graphs with large $k$. Discrete Appl. Math. 136,
55-61 (2004).
\bibitem{MasVol1996} Maschlanka, P., Volkmann, L.: Independence number
in $n$-extendable graphs. Discrete Math. 154, 167-178 (1996).
\bibitem{Plu1980} Plummer, M.D.: On $n$-extendable graphs. Discrete
Math. 31, 201-210 (1980).
\bibitem{Plu1986} Plummer, M.D.: Matching extension in bipartite
graphs. Congress. Numer. 54, 245-258 (1986).
\bibitem{Plu1994} Plummer, M.D.: Extending matchings in graphs: a
survey. Discrete Math. 127, 277-292 (1994).
\bibitem{Plu1996} Plummer, M.D.: Extending matchings in graphs: an update. Congress. Numer.
116, 3-32 (1996).
\bibitem{Plu2008} Plummer, M.D.: Recent progress in matching
extension. in: Gr\"{o}tschel, M., Katona, G.O.H. (eds.) Building
Bridges: Between Mathematics and Computer Science. Bolyai Society
Mathematical Studies, Volume 19, pp. 427-454 (2008).
\bibitem{Yu1993} Yu, Q.: Characterizations of various matching
extensions in graphs. Australas. J. of Combin. 7, 55-64 (1993).
\bibitem{YuLiu2009} Yu, Q., Liu, G.: Graph Factors and Matching
Extensions. Higher Education Press, Beijing (2009).
\bibitem{ZhaWanLou2007} Zhang, Z.-B., Wang, T., Lou, D.: Equivalence between extendibility and
factor-criticality. Ars Combin. 85, 279-285 (2007).
\bibitem{ZhaZha2001} Zhang, H., Zhang, F.: New lower bound on the number of perfect
matchings in fullerene graphs. J. Math. Chem. 30, 343-347 (2001).
\end{thebibliography}
\end{document}